\documentclass[4pt, oneside]{article}

\usepackage{yfonts}
\usepackage{amsmath,amssymb,amsthm}
\usepackage{units,stmaryrd}
\usepackage{graphicx}
\usepackage[all]{xy}
\newtheorem{theorem}{Theorem}[section]

\newtheorem{definition}[theorem]{Definition}
\newtheorem{lemm}[theorem]{Lemma}
\newtheorem{lemma}[theorem]{Lemma}
\newtheorem{example}[theorem]{Example}
\newtheorem{prop}[theorem]{Proposition}
\newtheorem{cor}[theorem]{Corollary}
\newtheorem{corollary}[theorem]{Corollary}

\def\x{{\xi}}

\usepackage{xspace}

\def\le{{\ell}}
\def\ex{e}

\def\fmodels{\xymatri\x{
\ar@{|=}[r]^{<\omega}&
}
}
\def\nmodels{\xymatri\x{
\ar@{|=}[r]^{N}&
}
}
\def\<{\left <}

\def\>{\right >}

\def\cbra{\left \{}
\def\cket{\right \}}

\DeclareSymbolFont{AMSb}{U}{msb}{m}{n}
\DeclareMathSymbol{\N}{\mathbin}{AMSb}{"4E}
\DeclareMathSymbol{\Z}{\mathbin}{AMSb}{"5A}
\DeclareMathSymbol{\R}{\mathbin}{AMSb}{"52}
\DeclareMathSymbol{\Q}{\mathbin}{AMSb}{"51}
\DeclareMathSymbol{\I}{\mathbin}{AMSb}{"49}
\DeclareMathSymbol{\C}{\mathbin}{AMSb}{"43}

\begin{document}

\title{Hyperations, Veblen progressions and transfinite iterations of ordinal functions}
\author{David Fern\'{a}ndez-Duque\footnote{Group for Logic, Language and Information, Universidad de Sevilla} \and Joost J. Joosten\footnote{Department of Logic, History and Philosophy of Science, University of Barcelona}}

\maketitle

\begin{abstract}
In this paper we introduce {\em hyperations} and {\em cohyperations}, which are forms of transfinite iteration of ordinal functions.

Hyperations are iterations of normal functions. Unlike iteration by pointwise convergence, hyperation preserves normality. The hyperation $\langle f^\xi \rangle_{\xi \in {\sf On}}$ of a normal function $f$ is a sequence of normal functions so that $f^0= {\sf id}$, $f^1 = f$ and for all $\alpha, \beta$ we have that $f^{\alpha + \beta} = f^\alpha f^\beta$. These conditions do not determine $f^\alpha$ uniquely; in addition, we require that $\langle f^\alpha\rangle_{\alpha\in{\sf On}}$ be minimal in an appropriate sense. We study hyperations systematically and show that they are a natural refinement of Veblen progressions. 

Next, we define \emph{cohyperations}, very similar to hyperations except that they are left-additive: given $\alpha, \beta$, $f^{\alpha + \beta}= f^\beta f^\alpha$. Cohyperations iterate initial functions which are functions that map initial segments to initial segments. We systematically study co-hyperations and see how they can be employed to define left inverses to hyperations.

Hyperations provide an alternative presentation of Veblen progressions and can be useful where a more fine-grained analysis of such sequences is called for. They are very amenable to algebraic manipulation and hence are convenient to work with. Cohyperations, meanwhile, give a novel way to describe slowly increasing functions as often appear, for example, in proof theory.
\end{abstract}

\section{Introduction}

Cantor famously discovered ordinal numbers \cite{cantor} when he needed to successively remove isolated points from a set of real numbers. Given a closed set $A\subseteq \mathbb R$, he considered the set $dA$ of all non-isolated points\footnote{More generally, if $A$ is an arbitrary subset of the reals, $dA$ is the set of limit points of $A$.} of $A$. The set $dA$ may in turn posses isolated points, so we may remove these by considering $d^2A$, and so on.

Evidently $d^{n+1}A\subseteq d^nA$ for all $n$, and we thus obtain a decreasing sequence
\[A\supseteq dA\supseteq ddA\supseteq\hdots \supseteq d^nA\supseteq\hdots\]
However, it may very well be that the set $\bigcap_{n<\omega}d^nA$ which one obtains at the end of this process still has isolated points. Thus it is convenient to continue the construction into the transfinite. To do this, let us denote the class of all ordinals by ${\sf On}$ and the class of limit ordinals by ${\sf Lim}$, and define
\begin{enumerate}
\item $d^0A=A$
\item $d^{\xi+1}A=dd^\xi A$ for all $\xi\in{\sf On}$
\item $d^\lambda A=\displaystyle\bigcap_{\xi<\lambda}d^\xi A$ for $\lambda\in{\sf Lim}$.
\end{enumerate}
 By cardinality considerations, given a closed set $A$ there must be an ordinal $\alpha$ such that $d^\alpha A=d^{\alpha+1}A$, thus effectively removing all `hereditarily isolated' points of $A$.

However, this is not the only setting in which one may wish to iterate a function transfinitely. When working with sets, one has the convenience of infinitary operations (such as intersection) for defining limit iterates, but this is not necessary. If $f:\mathbb N\to\mathbb N$, one may well define transfinite iterations of $f$ by diagonailzation. To be precise, let $\Lambda$ be a countable ordinal such that for each $\lambda<\Lambda$ we have assigned a fundamental sequence $\langle\lambda[n]\rangle_{n<\omega}$ with the property that $\lambda=\lim_{n\to\omega}\lambda[n]$.

Then, we may set
\begin{enumerate}
\item $f^0(n)=n$
\item $f^{\xi+1}(n)=f^\xi(f(n))$
\item $f^\lambda(n)=f^{\lambda[n]}(n)$.
\end{enumerate}
If one begins with $f(n)=n+1$, then one precisely obtains the Hardy hierarchy \cite{hardy}, while the sequence $\langle f^{\omega^\xi}\rangle_{\xi\in{\sf On}}$ gives the fast-growing Wainer hierarchy \cite{wainer}. Observe that for $\lambda\in{\sf Lim}$, the definition of $f^\lambda$ is somewhat ad-hoc, depending on our choice of fundamental sequence, while for successor ordinals the correspoinding iterate is determined by the previous ones. This is often the case; defining iterations at limit stages usually requires choosing a good option from many possible candidates, a phenomenon that we will also encounter.\\\\

In this paper we are interested, specifically, in transfinite iterations of functions $f:{\sf On}\to{\sf On}$. Here, as in the set-theoretic setting, one has infinitary operations at hand (such as taking suprema and infima), but as we shall see defining transfinite iterations is not unproblematic, especially if one wishes for iteration to preserve structural properties of $f$.

Perhaps the best-known example of a hierarchy of ordinal functions is given by Veblen progressions, introduced in \cite{Veblen:1908}. Recall that if $f=f_0$ is some {\em normal}\footnote{That is, strictly increasing and continuous. We use {\em continuous} in the sense of the order topology on ordinals, so that $f$ is continuous if $f(\lambda)=\lim_{\xi\to\lambda}f(\xi)$ for $\lambda\in{\sf Lim}$. Since normal functions are increasing, if $f$ is normal one may replace this condition by $f(\lambda)=\sup_{\xi<\lambda}f(\xi)$.} function and $\alpha$ an ordinal, $f_\alpha$ in the Veblen progression of $f$ is the normal function that enumerates in increasing order the class of ordinals
\[\{\xi\phantom{a}|\phantom{a}\forall \beta<\alpha,\,f_\beta(\xi)=\xi\};\]
we will discuss these in greater detail later. Veblen progressions may be used to give systems of ordinal notation below the Feferman-Sch\"ute ordinal, $\Gamma_0$ \cite{Feferman:1968,schute}. The method may be generalized to generate larger ordinals; see \cite{Crossley:1986} for an overview. 

However, Veblen progressions are not iterations in the sense we shall consider; for example, $f_2$ is not the same as $f_1\circ f_1$. Iterations of ordinal functions in a sense similar to our own are discussed in \cite{Girard:1984} (where a category-theoretic view of Veblen progressions is also given). There, the authors consider iterating monotone functions by taking limits:

\begin{definition}[Left transfinite iteration]\label{defli}
Let $f:{\sf On}\to{\sf On}$. We define the {\em left transfinite iteration of $f$} as the unique sequence of functions ${\rm TI}_l[f]=\langle f^\xi\rangle_{\xi\in{\sf On}}$ given by the following recursion:
\begin{enumerate}
\item $f^0  = {\sf id};$ 
\item $f^{\alpha+1} = f \circ f^\alpha;$
\item $f^\lambda \gamma =  \displaystyle\sup_{\alpha<\lambda} f^\alpha \gamma$ for $\lambda\in{\sf Lim}$.
\end{enumerate}
\end{definition}

As \cite{Girard:1984} already noted, ${\rm TI}_l$ preserves weak monotonicity but not normality. Further, if $f$ is continuous, then the sequence stagnates in the sense that $f^\xi=f^\omega$ for all $\xi>\omega$; we shall discuss this in detail in the next section. One of our main objectives is to describe a notion of transfinite iteration of normal functions which {\em does} preserve normality and does not stagnate. We shall define such iterations and call them {\em hyperations}.

Hyperations can be seen as a refinement of Veblen progressions. Much as with the Hardy and Wainer hierarchies, if $\langle f^\xi\rangle_{\xi\in{\sf On}}$ enumerates the hyperates of $f$ while $\langle f_\xi\rangle_{\xi\in{\sf On}}$ its Veblen progression, we have the relation
\[f^{\omega^\alpha}=f_\alpha.\]

Indeed, $f^\alpha$ may be defined completely in terms of Veblen progressions. However, we see at least three reasons why the hyperations viewpoint may often be useful:
\begin{enumerate}
\item it places the Veblen hierarchy in a new light providing an alternative presentation;

\item the algebraic structure that comes with hyperations is very convenient, facilitating applications, as the authors have found \cite{glpmodels,WellOrders};

\item hyperations naturally give rise to a uniform theory of certain well-behaved left-inverses of hyperations that we call {\em cohyperations}. 
\end{enumerate}

Cohyperations also come with a beautiful algebraic structure. We shall see that cohyperation is the proper notion for transfinitely iterating \emph{initial} functions, which are ordinal functions that map initial segments to initial segments. As such, cohyperations give rise to rather slow-growing ordinal functions.
In various parts of proof theory slow growing functions on the natural numbers play an important role (see for example \cite{Weiermann:2005,Weiermann:2009, rathjen}). The slowly-growing cohyperations may well give rise to other useful examples of slow-growing functions on the natural numbers, although at this point this is still speculative.
\medskip

As an academic curiosity by itself it is worthwhile to pursue a notion of transfinite iteration for ordinal functions outright, but the notions of hyperations and cohyperations were originally developed by the authors when working in the field of transfinite provability logics \cite{glpmodels,WellFoundedOrders, WellOrders} in order to extend many existing techniques into the transfinite setting. Provability logics, aside from being complex and fascinating on their own right, have been used by Beklemishev to give an ordinal analysis of Peano Arithmetic and related systems \cite{Beklemishev:2004}. We believe our efforts will lead to similar applications to much stronger systems, such as Predicative Analysis.\\\\

{\bf Plan of the paper.} In Section \ref{section:AdditivityVersusCoadditivity} we set the scene for our definition of transfinitely iterating normal functions and motivate the choices we made. In Section \ref{section:OrdinalArithmetic} we introduce the basic notions of ordinal arithmetic that we need for our study. 

Then, in Section \ref{section:Hyperations} we define hyperations as minimal transfinite iterations of a normal function. We establish the main properties of hyperations. In particular we provide two characterizations, a recursive definition and one characterization in terms of Veblen progressions. 

Next, Section \ref{ila} defines {\em left adjoints} to normal functions, which are particularly well-behaved inverses. Section \ref{exlo} then introduces exponentials and logarithms, which are important examples of ordinal functions in our setting. 

Section \ref{section:Cohyperations} introduces \emph{cohyperations}, another form of iterations which are dual to hyperations. Like with hyperations, we prove the essential properties of cohyperations and give a recursive characterization for them. Section \ref{exseq} then goes on to discuss exact sequences, which are very closely tied to cohyperations.

Finally, Section \ref{ih} relates hyperations and cohyperations by showing that the latter provide left-inverses for the former. As an important example we mention {\em hyperexponentials} and {\em hyperlogarithms}.

\section{Left vs. right additivity}\label{section:AdditivityVersusCoadditivity}
Let us restrict our attention in this section on how to transfinitely iterate \emph{normal} ordinal functions; that is, ordinal functions $f$ that are strictly monotone, i.e., $\alpha < \beta \Rightarrow f(\alpha ) < f(\beta)$, and continuous in the sense that $f(\lambda) = \displaystyle\lim_{\xi\to\lambda} f(\xi)$, for $\lambda\in{\sf Lim}$.

Let us start out by considering ${\rm TI}_l[f]$. If $f$ is normal, we can prove by an easy induction on $\beta$ that 
\begin{equation}\label{coadditivity}
f^{\alpha + \beta} = f^\beta \circ f^\alpha.
\end{equation}
We call this property \emph{left additivity}. We should remark that in this paper we often denote composition by simple juxtaposition, so we will often write left additivity as $f^{\alpha + \beta} = f^\beta f^\alpha$. We are also relaxed about writing parentheses around arguments, using $f(\xi)$ and $f\xi$ indistinctly.

Left additivity is a natural condition to demand of a sequence of transfinite iterations of a function $f$. Indeed, one would readily associate it with iteration and, for example, expect that
\[f^{\omega+1}=f\circ f^\omega;\]
\cite{simmons}, for example, uses left-additivity in defining iterations. Unfortunately, there are no non-trivial left-additive sequences of normal, or even injective, functions:

\begin{lemm}
Suppose that the sequence of functions $\langle f^\xi\rangle_{\xi\in{\sf On}}$ satisfies (\ref{coadditivity}) and $f$ is not the identity. Then, for any infinite ordinal $\xi$, $f^\xi$ is not injective.
\end{lemm}

\proof
Let $\gamma$ be such that $\gamma\not=f(\gamma)$.

Since $\xi$ is infinite we have that $1+\xi=\xi$, and we see that
\[
f^\xi f (\gamma) =  f^{1+\xi} (\gamma) = f^\xi (\gamma).
\]
\endproof

Thus, left additivity is enough to discard having injective iterations of injective functions, which is desirable for our purposes. A first attempt to avoid this problem is to work with a slightly different definition of transfinite iteration:

\begin{definition}[Right transfinite iteration]\label{defli}
Let $f:{\sf On}\to{\sf On}$. We define the {\em right transfinite iteration of $f$} as the unique sequence of functions ${\rm TI}_r[f]=\langle f^\xi\rangle_{\xi\in{\sf On}}$ given by the following recursion:
\begin{enumerate}
\item $f^0  = {\sf id};$ 
\item $f^{\alpha+1} = f^\alpha \circ f;$
\item $f^\lambda \gamma =  \displaystyle\sup_{\alpha<\lambda} f^\alpha \gamma$ for $\lambda\in{\sf Lim}$.
\end{enumerate}
\end{definition}

Note that the only difference with ${\rm TI}_l$ is Clause 2, where $f$ is now applied on the right. If $f$ is continuous, it is easy to see that under this definition each $f^\alpha$ is continuous and by induction on $\beta$ we can prove that
\begin{equation}\label{additivity}
f^{\alpha+\beta} = f^\alpha \circ f^\beta.
\end{equation}
We refer to this property as \emph{right additivity}.

In fact this condition does have advantages; (\ref{additivity}) alone is compatible with injectivity and even strict monotonicity of all iterates even if $f$ is not the identity. However, the recursion ${\rm TI}_r[f]$ is still not too useful when $f$ is a normal function, as it does not in general yield normal functions. For example, if we consider the normal ordinal function $f\equiv \alpha \mapsto 2^\alpha$ then our definition gives 
\[
f^\omega 0 = f^\omega 1 = \omega.
\]

What is more serious is that both ${\rm TI}_l$ and ${\rm TI}_r$ stabilize after $\omega$ iterations when $f$ is continuous:
\begin{lemm}\label{stagnates}
Let $f$ be a continuous function, and let ${\rm TI}_l[f]=\langle g^\xi\rangle_{\xi\in{\sf On}}$, ${\rm TI}_r[f]=\langle h^\xi\rangle_{\xi\in{\sf On}}$.

Then, for all $\xi\geq\omega$, $g^\omega=g^\xi=h^\xi$.
\end{lemm}

We will skip the proof which follows a simple induction, but it suffices to observe that
\[
g^{\omega+1}(\alpha)=fg^\omega (\alpha)=f\lim_{n \to\omega}f^n (\alpha)=\lim_{n \to\omega}f^{n+1} (\alpha)=g^\omega(\alpha);
\]
note that we are using the continuity of $f$ in the third equality.

Evidently, having notions of iteration which stagnate at $\omega$ is highly undesirable. One way out of this is to drop the limit condition from ${\rm TI}_r$ and replace it by (\ref{additivity}). Thus we arrive at a notion of transfinite iteration that we call \emph{weak hyperation}; a weak hyperation of a function $f$ is any sequence of functions $\langle g^\xi\rangle_{\xi\in{\sf On}}$ such that
\[
\begin{array}{llll}
g^0  & = & {\sf id}; & \\
g^{1} & = & f ; & \\
g^{\alpha + \beta} & =& g^\alpha \circ g^\beta.
\end{array}
\]
As we shall see, weak hyperations already have many desirable properties. However, they are not uniquely defined by the above condition; any normal function $f$ has many weak hyperations. A canonical candidate may be chosen by imposing a minimality constraint. We shall refer to this unique minimal weak hyperation as {\em the} {hyperation} of $f$. In Section \ref{section:Hyperations}, we will give a detailed definition as well as a thorough treatment of the basic properties of hyperations.

As mentioned, we shall prove that there is a very tight connection between Veblen progressions and hyperations and in particular that hyperations can be defined in terms of Veblen progressions, and vice-versa. However, one of the main advantages of our hyperations is that they allow for a uniform treatment of left-inverses with various desirable properties. We call these left-inverses \emph{cohyperations}.

Cohyperations transfinitely iterate {\em initial functions}, that is, functions mapping initial segments to initial segments. The definition is similar to that of hyperations, except that a cohyperation $\langle g^\xi\rangle_{\xi\in\mathsf{On}}$ is left-additive and pointwise {\em maximal} instead of minimal. These will be treated in Section \ref{section:Cohyperations}.

\section{Ordinal arithmetic}\label{section:OrdinalArithmetic}


Before continuing, let us give a brief review of some basic notions of ordinal arithmetic. We skip most proofs; for further details, we refer the reader to a text such as \cite{JechsBook} or \cite{PohlersBook}. 

Ordinals are canonical representatives for well-orders. The first infinite ordinal is as always denoted by $\omega$. Most operations on natural numbers can be extended to ordinal numbers, like addition, multiplication and exponentiation (see \cite{PohlersBook}). However, in the realm of ordinal arithmetic things often become more subtle. For example $1 + \omega = \omega \neq \omega + 1$, so addition is non-commutative. Other operations also differ considerably from ordinary arithmetic.

Fortunately, there are various similarities. In particular we have a form of subtraction available.
\begin{lemma}\label{theorem:BasicPropertiesOrdinalArithmetic}
\

\begin{enumerate}
\item Whenever $\zeta {<} \xi$, there exists a unique ordinal $\eta$ such that $\zeta + \eta = \xi$. We will denote this unique $\eta$ by $-\zeta + \xi$.

\item
Given $\eta>0$, there exist ordinals $\alpha,\beta$ such that $\eta = \alpha + \omega^{\beta}$. The value of $\beta$ is uniquely defined.
We will denote this unique $\beta$ by $\le \eta$.

\item
For all $\eta>0$, there exist $\alpha, \beta$ such that $\eta = \omega^{\alpha} + {\beta}$ and $\beta < \omega^{\alpha}+\beta$. The values of both $\alpha$ and $\beta$ are uniquely defined. We denote $\alpha$ by $L\eta$.
\end{enumerate}
\end{lemma}

One of the most useful way to represent ordinals is through their Cantor normal form (CNF):

\begin{theorem}[Cantor normal form theorem] \ \\
For each ordinal $\alpha$ there is a unique sequence of ordinals $\alpha_1\geq \ldots \geq \alpha_n$ such that 
\[
\alpha = \omega^{\alpha_1} + \ldots + \omega^{\alpha_n}.
\] 
\end{theorem}
We call a function $f$ \emph{increasing} if $\alpha < \beta$ implies $f(\alpha) < f(\beta)$. An ordinal function is called \emph{continuous} if $\lim_{\zeta\to\xi}f(\zeta) = f(\xi)$ for all limit ordinals $\xi$. Functions which are both increasing and continuous are called \emph{normal}. 

It is not hard to see that each normal function has an unbounded set of fixpoints. For example, the first fixpoint of the function $\varphi: \xi\mapsto \omega^\xi$ is
\[\sup \{  \omega, \omega^{\omega},\omega^{\omega^{\omega}}, \ldots \}\]
and is denoted $\varepsilon_0$. Clearly for these fixpoints, CNFs are not too informative as, for example, $\varepsilon_0 = \omega^{\varepsilon_0}$. To remedy this, one may use notations and normal forms that are slightly more informative and can represent the fixpoints of normal functions, as is the case of Veblen Normal Forms (VNFs).

In his seminal paper \cite{Veblen:1908}, Veblen considered for each normal function $f$ its derivative $f'$ that enumerates the fixpoints of $f$. 
%
If $f$ is a normal function, then the image of $f$ --which we shall denote by $F$-- is a closed (under taking suprema) unbounded set. Likewise the function that enumerates a closed unbounded set is continuous.  
For $f$ a normal function, we define $F'$ to be the image of $f'$ and we extend this transfinitely by setting
\[
\begin{array}{llll}
F_0&:=&F;&\\
F_{\alpha+1} &:=&(F_{\alpha})';& \\
F_{\lambda} &:=& \displaystyle\bigcap_{\alpha<\lambda} F_{\alpha}& \mbox{ for limit $\lambda$}.
\end{array}
\]
We then define $f_\lambda$ to be the function that enumerates $F_{\lambda}$ and write ${\rm Veb}[f]=\langle f_\lambda\rangle_{\lambda\in{\sf On}}$.

By taking $\varphi := \xi\mapsto\omega^{\xi}$, one obtains the familiar Veblen functions $\varphi_{\alpha}$. However, Veblen progressions can be constructed out of {\em any} normal function.

We can often write an ordinal $\omega^{\alpha}$ in many ways as $\varphi_{\xi}(\eta)$. However, if we require that $\eta < \varphi_{\xi}(\eta)$, then both $\xi$ and $\eta$ are uniquely determined. In other words, for every $\alpha$, there are unique ordinals $\eta,\xi$ such that
\begin{enumerate}
\item $\omega^{\alpha} = \varphi_{\xi}(\eta)$ and
\item$\eta < \varphi_{\xi}(\eta)$.
\end{enumerate}
Combining this fact with the CNF Theorem one obtains {\em Veblen Normal Forms} for ordinals.
\begin{theorem}[Veblen Normal Form Theorem]\label{vnft}
For all $\alpha$ there exist unique $\alpha_1, \beta_1, \ldots ,\alpha_n,\beta_n$ ($n\geq 0$) such that 
\begin{enumerate}
\item
$\alpha = \varphi_{\alpha_1}(\beta_1) + \ldots + \varphi_{\alpha_n}(\beta_n)$,

\item
$\varphi_{\alpha_i}(\beta_i) \geq \varphi_{\alpha_{i+1}}(\beta_{i+1})$ for $i<n$,

\item
$\beta_i < \varphi_{\alpha_i}(\beta_i)$ for $i\leq n$.
\end{enumerate}
\end{theorem}
\noindent
Note that $\alpha_i\geq\alpha_{i+1}$ does not in general hold in the VNF of $\alpha$; for example,
\[\omega^{\varepsilon_0+1} + \varepsilon_0 = \varphi_0(\varphi_{\varphi_0(0)}(0)+\varphi_0(0)) + \varphi_{\varphi_0(0)}(0).\]

It is convenient to also define Veblen functions ``from below''. For this we use the following useful fact:
\begin{lemm}\label{fixpoints}
Let $f$ be a normal function and $\xi$ an ordinal. Then, $\zeta=\lim_{n\to\omega}f^n(\xi)$ is the least fixed point of $f$ greater or equal to $\xi$, that is, it is the least ordinal such that \[\xi\leq\zeta=f(\zeta).\]
\end{lemm}

\proof
By Lemma \ref{stagnates}, $f\circ \lim_{n\to\omega}f^n=\lim_{n\to\omega}f^n$, so that $\zeta$ is a fixed point of $f$. 

Now, in order to see that it is the least such fixed point, let us assume that $f(\xi)\not=\xi$ (otherwise $\zeta=\xi$ and there is nothing to prove). Then, by an easy induction it follows that $f^n(\xi)<f^{n+1}(\xi)$ for all $n$.

Pick any $\eta\in(\xi,\zeta)$ and let $N$ be the greatest number such that $f^N(\xi)<\eta$; such an $N$ exists since $\langle f^n(\xi)\rangle_{\xi\in{\sf On}}$ is an increasing sequence converging to $\zeta$.

But then $f^{N+1}(\eta)<f(\eta)$, and by maximality of $N$ we also have that $\eta<f(\eta)$, i.e., $\eta$ is not a fixed point of $f$.
\endproof

As a corollary we also get that $\lim_{n\to\omega}f^n(\xi+1)$ gives the smallest fixpoint of $f$ which is {\em greater} than $\xi$. This observation is an essential ingredient in giving a recursive definition of Veblen progressions:

Below, $f^n_\xi$ denotes $(f_\xi)^n$, which is typically {\em not} equal to $(f^n)_\xi.$

\begin{theorem}\label{vebrec}
Let $f$ be a normal function an let ${\rm Veb}[f]=\langle f_\xi\rangle_{\xi\in{\sf On}}$.

Then,
\begin{enumerate}
\item for all $\xi\in{\sf On}$, $f_{\xi+1}(0)=\displaystyle\lim_{n\to\omega}f^n_{\xi}(0)$

\item for all $\xi,\alpha\in{\sf On}$, $f_{\xi+1}(\alpha+1)=\displaystyle\lim_{n\to\omega}f^n_{\xi}(f_{\xi+1}(\alpha)+1)$

\item for all $\lambda\in{\sf Lim}$, $f_{\lambda}(0)=\displaystyle\lim_{\xi\to\lambda}f_{\xi}(0)$

\item for all $\lambda\in{\sf Lim}$ and $\alpha\in{\sf On}$, $f_{\lambda}(\alpha+1)=\displaystyle\lim_{\xi\to\lambda}f_{\xi}(f_{\lambda}(\alpha)+1)$

\item for all $\xi\in{\sf On}$ and $\alpha\in{\sf Lim}$, $f_{\xi}(\alpha)=\displaystyle\lim_{\beta\to\alpha}f_{\xi}(\beta)$.

\end{enumerate}
\end{theorem}

Although we will not prove this result, note that Item 2 follows from Lemma \ref{fixpoints}; we have that $\displaystyle\lim_{n\to\omega}f^n_{\xi}(f_{\xi+1}(\alpha)+1)$ is the smallest fixpoint of $f_\xi$ greater than $f_{\xi+1}(\alpha)$, which is precisely the meaning of $f_{\xi+1}(\alpha+1)$.\\\\

We are now ready to begin the study of hyperations.

\section{Hyperations}\label{section:Hyperations}


{\em Hyperation} is a form of transfinite iteration of normal functions. If $f$ is a normal function, we wish to define the hyperates $\langle f^\xi\rangle_{\xi \in {\sf On}}$ of $f$, where $\xi$ is an arbitrary ordinal, in such a way that $f^\xi$ is always a normal function and, in accordance to our considerations from Section \ref{section:AdditivityVersusCoadditivity}, $f^{\xi+\zeta}=f^\xi f^\zeta$. 


\subsection{Weak hyperations and hyperations}

Recall that $\mathsf{On}$ denotes the class of all ordinals and $\mathsf{Lim}$ the class of limit ordinals.

\begin{definition}[Weak hyperation]
Let $f$ be a normal function and $\Lambda$ be either an ordinal or the class of all ordinals.

A {\em weak hyperation} of $f$ is a family of normal functions $\langle g^{\x}\rangle_{\xi<\Lambda}$ such that
\begin{enumerate}
\item $g^0{\x}={\x}$ for all ${\x}$,
\item $g^1=f$,
\item $g^{{\x}+\zeta}=g^{\x} g^\zeta$ whenever $\xi+\zeta<\lambda$.
\end{enumerate}
\end{definition}

In general, a normal funcion has many weak hyperations; the values of $g^{\omega}$ are not uniquely defined, given that $\omega$ cannot be written as a sum of smaller ordinals. The same can be said of all $g^{\omega^\delta}$ for $\delta > 1$. However, weak hyperations are all rather well-behaved, as witnessed by the following lemma:

\begin{lemma}[Properties of weak hyperations]\label{prophyp}
Any weak hyperation $\langle g^{\x}\rangle_{\xi\in{\sf On}}$ has the following properties:
\begin{enumerate}
\item if ${\x}<\zeta$ then $g^{\x}\alpha\leq g^\zeta\alpha$,
\item if ${\x}+\zeta=\zeta$ then $g^{\x} g^\zeta=g^\zeta$,
\item if $\eta<\omega^\rho$ then
\[
g^\eta(g^{\omega^\rho}({\x})+1)\leq g^{\omega^\rho}({\x}+1).
\]
\end{enumerate}
\end{lemma}

\proof\
\paragraph{1.} If ${\x}<\zeta$ and $\alpha$ is any ordinal, then $g^\zeta \alpha=g^{\x} g^{-{\x}+\zeta}\alpha$; but $g^{-{\x}+\zeta}$ is a normal function so $g^{-{\x}+\zeta}\alpha\geq\alpha$, thus $g^{\x} g^{-{\x}+\zeta}\alpha\geq g^{\x}\alpha$.

\paragraph{2.}By right additivity we have that $g^{\zeta}=g^{{\x}+\zeta}=g^{\x} g^\zeta$.

\paragraph{3.} First note  that
\[g^{\omega^\rho}({\x})+1\leq g^{\omega^\rho}({\x}+1)\]
because $g^{\omega^\rho}$ is increasing.

Meanwhile, $\eta<\omega^\rho$ so $\eta+\omega^\rho=\omega^\rho$, hence
\[g^\eta(g^{\omega^\rho}({\x})+1)\leq
g^\eta g^{\omega^\rho}({\x}+1)=g^{\omega^\rho}({\x}+1).\]
\endproof

\begin{corollary}
Let $f^{\alpha}$ be a weak hyperation. Then all values of $f^{\zeta}$ are fixpoints of $f^{\xi}$ whenever $\xi + \zeta = \zeta$.
\end{corollary}

\begin{proof}
Immediate from Lemma \ref{prophyp}.2.
\end{proof}

The following lemma tells us that we can weaken the requirement of additivity in the definition of weak hyperations to a few special cases.

\begin{lemma}\label{char}
Let $\vec g=\langle g^{\x}\rangle_{\xi<\Lambda}$ be a family of normal functions.

Then, the following are equivalent:
\begin{enumerate}
\item $\vec g$ is a weak hyperation of $f$.
\item $\vec g$ has the following properties:
\begin{enumerate}
\item $g^0{\x}={\x}$ for all ${\x}$,
\item $g^1=f$,
\item $g^{\omega^\rho+{\x}}=g^{\omega^\rho}g^{\x}$ whenever $\x < \omega^{\rho} + \x<\Lambda$,
\item $g^{\omega^\delta}=g^{\omega^\delta} g^{\omega^\rho}$ whenever $\omega^\delta<\omega^\rho<\Lambda$.
\end{enumerate}
\end{enumerate}
\end{lemma}

\proof
Clearly the first item implies the second, particularly in view of Lemma \ref{prophyp}.

To see that the second implies the first, it suffices to show that $\vec g$ is right additive, that is, $ g^{\x + \zeta} = g^{\x} g^{\zeta}$ whenever $\xi+\zeta<\Lambda$.

We proceed by induction on ${\x}$. For the base case, note that
\[g^{0+\zeta}=g^\zeta=g^0g^\zeta\]
independently of $\zeta$, given that $g^0$ is the identity.\\

For the inductive step, suppose $\xi+\zeta<\Lambda$ and ${\x}>0$. The case for $\zeta=0$ is analogous to the case ${\x}=0$, so we assume $\zeta>0$.

Write ${\x}={\x}'+\omega^\alpha$ (so that ${\x}'<{\x}$) and $\zeta=\omega^\beta+\zeta'$, with $\zeta'<\zeta$. By induction on $\xi'<{\x}$, we assume that $g^{{\x}'+\vartheta}=g^{{\x}'}g^\vartheta$ whenver $\x'+\vartheta<\Lambda$. In particular,
\[
g^{\x}=g^{{\x}'+\omega^\alpha}=g^{{\x}'}g^{\omega^\alpha}.
\]

Now we consider two cases. First, assume $\alpha<\beta$. In this case, $\omega^\alpha+\omega^\beta=\omega^\beta$. Meanwhile, by Property (d), $g^{\omega^\alpha}g^{\omega^\beta}=g^{\omega^\beta}$ and 
we have that
\[g^{\x} g^{\zeta}=g^{{\x}'}g^{\omega^\alpha}g^{\omega^\beta}g^{\zeta' }=g^{{\x}'}g^{\omega^\beta}g^{\zeta' }=g^{{\x}'}g^\zeta\stackrel{\mathrm{IH}}=g^{{\x}'+\zeta}=g^{{\x}+\zeta}.\]

Now, if $\alpha\geq\beta$, we have that ${\zeta}<\omega^\alpha+{\zeta}$ and thus by using Property (c) we get that 
\[g^{\x} g^\zeta=g^{{\x}'}g^{\omega^\alpha}g^\zeta=g^{{\x}'} g^{\omega^\alpha+\zeta}\stackrel{\mathrm{IH}}=g^{{\x}'+\omega^\alpha+\zeta}=g^{{\x}+\zeta},\]
as desired.
\endproof

\begin{definition}[Hyperation]
A weak hyperation $\langle g^\xi\rangle_{\xi\in{\sf On}}$ of $f$ is {\em minimal} if it has the property that, whenever $\langle h^{\x}\rangle_{\xi\in{\sf On}}$ is a weak hyperation of $f$ and ${\x},\zeta$ are ordinals, then $g^{\x}\zeta\leq h^{\x}\zeta$.

If $f$ has a minimal weak hyperation, we call it {\em the hyperation} of $f$ and denote it ${\rm Hyp}[f]$.
\end{definition}

Note that hyperations, if they exist, are unique. It is also the case that they always exist; for this, we will establish an explicit recursion scheme to compute the unique hyperation of a normal function $f$. In particular, we see that the inequality we proved in Lemma \ref{prophyp}.3 is optimal as equality is attained in the limit for hyperations.

\begin{theorem}
Every normal function $f$ has a unique hyperation ${\rm Hyp}[f]=\langle f^\xi\rangle_{\xi\in{\sf On}}$ and it is given by the following recursion:
\begin{enumerate}
\item $f^0{\x}={\x}$,
\item $f^1=f$,
\item $f^{\omega^\rho+{\x}}=f^{\omega^\rho}f^{\x}$, where $0<{\x}<\omega^\rho+{\x}$,
\item $f^{\omega^\rho}0=\lim_{\zeta\to\omega^\rho}f^\zeta 0$ for $\rho >0$,
\item $f^{\omega^\rho}({\x}+1)=\lim_{\zeta\to\omega^\rho}f^{\zeta}(f^{\omega^\rho}({\x})+1)$ for $\rho >0$,
\item $f^{\omega^\rho}{\x}=\lim_{\zeta\to{\x}}f^{\omega^\rho}\zeta$ for ${\x}\in\mathsf{Lim}$ and $\rho >0$.
\end{enumerate}
\end{theorem}

\proof
By an easy induction on $\alpha$ and an auxiliary induction on $\beta$ we readily see that $f^\alpha(\beta)$ is well-defined for all $\alpha$ and $\beta$. Thus, we must show  that
\begin{enumerate}
\item $\langle f^\xi\rangle_{\xi\in{\sf On}}$ is a weak hyperation and
\item if $\langle g^\xi\rangle_{\xi\in{\sf On}}$ is another weak hyperation, then $g^{\x}\zeta\geq f^{\x}\zeta$ for all ${\x},\zeta$.
\end{enumerate}

For the first point, we first need to show that $f^\x$ always normal, i.e., strictly increasing and continuous.

This is clear whenever $\x$ is of the form $\omega^\rho+\x'$ with $\omega^\rho+\x' > \x'>0$, since compositions of normal functions are always normal. Thus we may assume $\x=\omega^\rho$.

The continuity of $f^\x $ is evident from Clause 6. To check that $f^\x$ is increasing it suffices to verify that $f^\x\zeta<f^\x(\zeta+1)$ for all $\zeta$, but by induction, $f^{\eta}$ is normal for all $\eta<\x$, thus $f^{\eta}(f^{\omega^\rho}({\zeta})+1)\geq f^{\omega^\rho}{\zeta}+1$ and
\[
\lim_{\eta\to \omega^\rho}f^{\eta}(f^{\omega^\rho}({\zeta})+1)\geq f^{\omega^\rho}({\zeta})+1>f^{\omega^\rho}{(\zeta)}.
\]

Hence it remains to show that $\langle f^\xi\rangle_{\xi\in{\sf On}}$ is additive. To do this, we will prove by induction on $\Lambda$ that $\langle f^\xi\rangle_{\xi<\Lambda}$ is a weak hyperation.

In view of Lemma \ref{char}, it suffices to check that $f^{\omega^\delta}f^{\omega^\rho} \xi=f^{\omega^\rho}\xi$ whenever $\xi\in{\sf On}$ and $\omega^\delta<\omega^\rho<\Lambda$. By induction we may assume that $\langle f^\eta\rangle_{\eta<\omega^\rho}$ is a weak hyperation.

We then proceed case by case depending on whether $\xi$ is zero, a successor or a limit ordinal. Here we shall only consider the case ${\x}=\zeta+1$, which is the most involved:
\[
\begin{array}{lcl}
f^{\omega^\delta}f^{\omega^\rho}{\x} &=&f^{\omega^\delta}\displaystyle\lim_{\eta\to\omega^\rho}f^{\eta}(f^{\omega^\rho}(\zeta)+1)\\
&=&\displaystyle\lim_{\eta\to\omega^\rho}f^{\omega^\delta}f^{\eta}(f^{\omega^\rho}(\zeta)+1)\\
&\stackrel{\mathrm{IH}}=&\displaystyle\lim_{\eta\to\omega^\rho}f^{\omega^\delta+\eta}(f^{\omega^\rho}(\zeta)+1)\\
&=&\displaystyle\lim_{\eta\to\omega^\rho}f^{\eta}(f^{\omega^\rho}(\zeta)+1)\\
&=&f^{\omega^\rho}{\x}.
\end{array}
\]

Now we must check that if $\vec g$ is a weak hyperation of $f$, $f^{\x}\zeta\leq g^{\x}\zeta$ for all ${\x},\zeta$. We can assume ${\x}>1$, for otherwise there is nothing to prove.

We then establish the claim by induction on $\xi$ with a subsidiary induction on $\zeta$, assuming it for ${\x}',\zeta',$ where either ($i$) ${\x}'<{\x}$ or ($ii$) ${\x}'={\x}$ and $\zeta'<\zeta$.

As before, this proceeds case-by-case, where each argument is similar.
For example, assume $\zeta=0$ and ${\x}\in\mathsf{Lim}$. We then have that $ f^\eta 0\leq g^{\eta}0$ for $\eta<{\x}$; because $ g^\eta 0\leq g^{\x}0$ (Lemma \ref{prophyp}), it follows that
\[f^{\x} 0=\lim_{\eta\to{\x}}f^\eta 0\leq \lim_{\eta\to{\x}}g^\eta 0\leq g^{\x} 0.\]

Once again, the most involved case is where ${\x}=\omega^\rho$ and $\zeta=\alpha+1.$ Here we must observe that
\[
f^{\eta}(f^{\omega^\rho}(\alpha)+1)\leq g^{\eta}(g^{\omega^\rho}(\alpha)+1).
\]
This follows from our induction hypothesis since
\[
f^{\eta}(f^{\omega^\rho}(\alpha)+1)\leq f^{\eta}(g^{\omega^\rho}(\alpha)+1)\leq g^{\eta}(g^{\omega^\rho}(\alpha)+1),
\]
where the first inequality follows by induction on $\alpha<\zeta$ and the second by induction on $\eta<\xi$.

It follows that
\[f^{\omega^\rho}\zeta=\lim_{\eta\to\omega^\rho}f^{\eta}(f^{\omega^\rho}\alpha+1)\leq \lim_{\eta\to\omega^\rho}g^{\eta}(g^{\omega^\rho}\alpha+1)\leq g^{\omega^\rho}\zeta,\]
as claimed.\endproof

Let us include an easy example to see how the recursion works.
\begin{example}\label{Example:eOmega2}
 Recall that we defined $\varphi(\alpha)=\omega^\alpha$, and that the fixpoints of $\varphi$  are the {\em epsilon numbers} $\langle \varepsilon_\xi\rangle_{\xi\in{\sf On}}$.

Let $\langle\varphi^\xi\rangle_{\xi\in{\sf On}}={\rm Hyp}[\varphi]$.

Then, we have that $\varphi^\omega 1 = \varepsilon_1$, that is, $\varphi^\omega 1$ equals the second-smallest fixpoint of the map $\alpha\mapsto \omega^\alpha$. 
\end{example}

\begin{proof}
We first compute
\[\varphi^\omega 0= \lim_{n\to \omega} \varphi^n (0) = \lim_{n\to \omega} \varphi^n (0).\]
By Lemma \ref{fixpoints}, this limit defines the first fixpoint of $\alpha\mapsto \omega^\alpha$, normally denoted $\varepsilon_0$. Likewise, we compute
\[\varphi^\omega 1= \lim_{n\to \omega} \varphi^n (\varphi^\omega(0) + 1) =\lim_{n\to \omega} \varphi^n (\varepsilon_0 + 1).\]
Again by Lemma \ref{fixpoints}, this is the smallest fixpoint of $\varphi$ which is larger than $\varepsilon_0$, whence it is the next epsilon number, i.e. $\varepsilon_1$, as was to be shown.
\end{proof}

This is a special case of the more general equality \[\varphi^{\omega}\xi=\varphi_1\xi=\varepsilon_\xi;\]
we will return to this in the next section. It should also be noted that hyperations of hyperations behave as one would expect:

\begin{lemma}
If $f$ is any normal function and $\xi,\zeta$ are ordinals then $(f^\xi)^\zeta=f^{\xi\cdot\zeta}$, where $\langle f^\eta\rangle_{\eta\in{\sf On}}={\rm Hyp}[f]$ and $\langle (f^\xi)^\eta\rangle_{\eta\in{\sf On}}={\rm Hyp}[f^\xi]$.
\end{lemma}

\proof
We show that $(f^\xi)^\zeta\alpha=f^{\xi\cdot\zeta}\alpha$ by induction on $\zeta$ with an auxiliary induction on $\alpha$.

For $\zeta = 0$ there is nothing to prove and, if $\zeta=\gamma+\delta$ with $\delta < \zeta$, then
\[
(f^\xi)^{\gamma+\delta}=(f^\xi)^\gamma(f^\xi)^\delta\stackrel{\mathrm{IH}}=f^{\xi\cdot\gamma}f^{\xi\cdot\delta}=f^{\xi\cdot\gamma+\xi\cdot\delta}=f^{\xi\cdot(\gamma+\delta)}.
\]

Otherwise, $\zeta=\omega^\gamma$ for some $\gamma$. We note that in this case, $\xi\cdot\omega^\gamma=\omega^\delta$ for some $\delta$.

Then we must consider three cases, depending on whether $\alpha$ is zero, a successor or a limit ordinal.

In case $\alpha =0$ we see that 
\[
(f^\xi)^\zeta 0 = \lim_{\eta\to \zeta} (f^\xi)^\eta 0 
\stackrel{\mathrm{IH}}= \lim_{\eta\to \zeta} f^{\xi \cdot \eta} 0 = \lim_{\nu \to \omega^\delta} f^{\nu} 0 = f^{\omega^\delta} 0.
\]
Note that $\omega^\delta= \lim_{\eta\to\omega^\gamma}\xi\cdot\eta$
since multiplication is continuous on the right-hand argument.

In the successor case,
we write $\alpha=\beta+1$. Then,
\[
(f^\xi)^{\omega^\gamma}\alpha=\lim_{\eta\to\omega^\gamma}(f^\xi)^\eta((f^\xi)^{\omega^\gamma}(\beta)+1)\stackrel{\mathrm{IH}}=\lim_{\eta\to\omega^\delta}f^\eta(f^{\omega^\delta}(\beta)+1)=f^{\omega^\delta}\alpha,
\]
where in the inductive step we are simultaneously using $(f^\xi)^{\omega^\gamma}\beta=f^{\omega^\delta}\beta$ by induction on $\beta<\alpha$ and $(f^\xi)^\eta=f^{\xi\cdot\eta}$ by induction on $\eta<\zeta=\omega^\gamma$. 

The limit case follows from
\[
(f^\xi)^{\omega^\gamma}\alpha=\lim_{\beta\to\alpha}(f^\xi)^{\omega^\gamma}\beta\stackrel{\mathrm{IH}}=\lim_{\beta\to\alpha}f^{\omega^\delta}\beta=f^{\omega^\delta}\alpha .
\]
\endproof

\subsection{Hyperations and Veblen progressions}\label{section:HyperationsAndVeblenProgressions}

We shall now see that Veblen progressions can be naturally embedded in our hyperations. The steps in the Veblen progressions correspond to specific iterates; as such, hyperations form a natural refinement of Veblen progressions.

Troughout this subsection, given a normal function $f$, ${\rm Veb}[f]=\langle f_{\alpha}\rangle_{\alpha\in{\sf On}}$ and ${\rm Hyp}[f]=\langle f^{\alpha}\rangle_{\alpha\in{\sf On}}$.
\begin{lemma}\label{theorem:HyperExponentialsAndVeblen}
Given a normal function $f$ and an ordinal $\alpha$, we have that $f^{\omega^\alpha} =f_\alpha$.
\end{lemma}

\proof
We prove the lemma by induction on $\alpha$.
By definition we have $f^{\omega^0}=f_0$ which settles the base case.

For $\alpha + 1$
we proceed to prove that  $f^{\omega^{\alpha+1}} \beta =f_{\alpha+1} \beta$ by a subsidiary induction on $\beta$  considering several cases.
For $\beta=0$ we see that
\[
f^{\omega^{\alpha+1}}0 = 
\lim_{n\to \omega} f^{\omega^\alpha \cdot n}0 = 
\lim_{n\to \omega} (f^{\omega^\alpha})^n 0 \stackrel{\mathrm{IH}}=
\lim_{n\to \omega} (f_\alpha)^n 0.
\]
But by Theorem \ref{vebrec}, $\lim_{n\to \omega} (f_\alpha)^n 0 = f_{\alpha+1}0$.

Likewise, for $\beta +1$ we see that
\begin{align*}
f^{\omega^{\alpha+1}}(\beta+1) &=
\lim_{n\to \omega}(f^{\omega^\alpha})^n(f^{\omega^{\alpha+1}}(\beta)+1)\\ &\stackrel{\mathrm{IH}}= \lim_{n\to \omega} (f_\alpha)^n (f_{\alpha+1}(\beta) +1)\\
&=f_{\alpha+1}(\beta +1).
\end{align*}

The cases for limit $\beta$ or limit $\alpha$ are similar and we omit them here.
\endproof

\begin{corollary}\label{theorem:hyperationsAndIteratedVeblen}
If
\[{\alpha}=\omega^{\alpha_1}+\hdots+\omega^{\alpha_n},\]
then
\[f^{\alpha}\zeta=f_{\alpha_1}\hdots f_{\alpha_n}(\zeta)\]
for any ordinal $\zeta$.
\end{corollary}
\begin{proof}
Directly from Theorem \ref{theorem:HyperExponentialsAndVeblen} and the fact that we have $f^{\xi + \zeta} = f^\xi f^\zeta$.
\end{proof}

We thus see that hyperations define Veblen progressions. As such hyperations are natural refinements of them and a Veblen progression ${\rm Veb}[f]$ comes with a corresponding refinement ${\rm Hyp}[f]$. We shall now see that Veblen progressions in their turn uniquely determine weak hyperations.

\begin{theorem}\label{theorem:VeblenDefinesHyperation}
Let $\vec g=\langle g^{\xi}\rangle_{\xi\in{\sf On}}$ be a weak hyperation of a normal function $f$. If we moreover have that $g^{\omega^\alpha} =f_\alpha$ for each $\alpha$ then $\vec g={\rm Hyp}[f]$.
\end{theorem}

\begin{proof}
Write $\xi = \omega^{\xi_1} + \ldots + \omega^{\xi_n}$ in CNF. By additivity we see that 
\[
g^{\xi} = g^{\omega^{\xi_1}} \ldots g^{\omega^{\xi_n}},
\]
which by assumption equals $f_{\xi_1}\ldots f_{\xi_n}$. By Corollary \ref{theorem:hyperationsAndIteratedVeblen}, it follows that $g^\xi=f^\xi$, as claimed.
\end{proof}

\section{Initial functions and left adjoints}\label{ila}

Normal functions are typically not surjective and hence non-invertible. However, being injective, they are always left-invertible. Further, if $f$ is normal and $gf$ is the identity, $g$ is usually not normal. {Initial functions}, meanwhile, provide natural left-inverses for normal functions, in particular when they are also {\em left adjoints}.

We will say a function $g$ is {\em initial} if, whenever $I$ is an initial segment (i.e., of the form $[0,\beta)$ for some $\beta$), then $f(I)$ is an initial segment.

\begin{lemm}
If $f$ is initial, then $f{\x}\leq{\x}$ for every ordinal ${\x}$.
\end{lemm}

\proof
By a simple induction on $\xi$. If $f$ is initial we have that
\[f[0,\xi)=[0,\beta)\]
for some $\beta$; by induction on $\xi$, if $\zeta<\xi$ it follows that $f(\zeta)\leq\zeta$, and thus $\beta\leq \sup_{\zeta<\xi}(\zeta+1)=\xi$.

Further,
\[f[0,\xi]=f[0,\xi)\cup\cbra f(\xi)\cket=[0,\beta)\cup\cbra f(\xi)\cket\]
must also be an initial segment, from which it follows that $f(\xi)\leq\beta\leq \xi$.
\endproof

Let $f$ be a normal function. Then, $g$ is a {\em left adjoint} for $f$ if, for all ordinals $\alpha,\beta$,
\begin{equation}\label{laone}
\alpha= f(\beta)\Rightarrow g(\alpha)=\beta
\end{equation}
and
\begin{equation}\label{latwo}
\alpha< f(\beta)\Rightarrow g(\alpha)<\beta.
\end{equation}\\\\

Not all normal functions have left adjoints:

\begin{lemm}\label{iszero}
If a normal function $f$ has a left adjoint, then $f(0)=0$.
\end{lemm}

\proof
This follows from the fact that, if $\alpha<f(0)$ and $g$ is a left adjoint to $f$, then $g(\alpha)<0$, which is absurd; thus there can be no such $\alpha$, that is, $f(0)=0$.
\endproof

However, this is the only condition we need to have left adjoints of normal functions, which are always initial:

\begin{lemm}\label{isanadj}
Let $f$ be a normal function with $f(0)=0$.

Then, $f$ has an initial left adjoint $g$. Further, if $h$ is any left adjoint for $f$, then $h$ is initial.
\end{lemm}

\proof
The existence of a left adjoint is easy; simply consider $g$ defined by $g(\beta)=\alpha$ if $\beta=f(\alpha)$, $g(\beta)=0$ if there is no such $\alpha$.

Now, if $h$ is a left adjoint for $f$, we claim that for all $\beta$, $h[0,\beta]=f^{-1}[0,\beta]$, which is clearly an initial segment since $f$ is increasing. This follows by induction on $\beta$.

First note that
\[h[0,\beta)=\bigcup_{\delta<\beta}h[0,\delta]\stackrel{\mathrm{IH}}=\bigcup_{\delta<\beta}f^{-1}[0,\delta]=f^{-1}[0,\beta).\]
Thus, $h[0,\beta]=f^{-1}[0,\beta)\cup \cbra h(\beta)\cket$.

Now, if $\beta$ lies in the range of $f$, $\beta=f(\gamma)$ implies that $f^{-1}\cbra\beta\cket=\cbra\gamma\cket$, while $h(\beta)=\gamma$ because $h$ is a left adjoint of $f$. It follows that $\cbra h(\beta)\cket=f^{-1}\cbra\beta\cket$, so that
\[h[0,\beta]=f^{-1}[0,\beta)\cup f^{-1}\cbra\beta\cket=f^{-1}[0,\beta].\] 

If $\beta$ does not lie in the range of $f$, let $\gamma$ be the least ordinal such that $\beta<f(\gamma)$. Then, $h(\beta)<\gamma$, so that by minimality of $\gamma$, $fh(\beta)<\beta$ and thus $h(\beta)\in f^{-1}[0,\beta)$; therefore, 
\[h[0,\beta]=h[0,\beta)\cup\cbra h(\beta)\cket=f^{-1}[0,\beta)=f^{-1}[0,\beta].\qedhere\]
\endproof

Thus initial functions are natural inverses to normal functions. Because of this, later in the text we shall turn our attention to iterations of initial functions. But first, let us introduce some important examples.

\section{Exponentials and logarithms}\label{exlo}

In this section we shall define some particularly natural and interesting ordinal functions from the point of view of our framework. After deriving some basic properties we see how to obtain an ordinal notation system from these new functions.

\subsection{Exponentials and logarithms}

By Lemma \ref{iszero}, we know that $\varphi$ (the $\omega$-based exponential) does not have any left adjoint. However, we can introduce a mild modification which does.

Set $e(\xi)=-1+\omega^\xi$. Note that $e(\xi)=\omega^\xi$ unless $\xi=0$, in which case $e(0)=0$. The advantage of $e$ over $\varphi$, as mentioned, is the existence of left adjoints. We mention two:

%
%

\begin{enumerate}
\item The {\em left-logarithm}, {\em first exponent} or {\em greatest exponent} $L$, where $L(\xi)$ is the unique ordinal $\alpha$ so that $\xi=\omega^\alpha+\beta$ with $\beta<\omega^\alpha+\beta$. We set $L(0)=0$. We remark that, given $\xi,\zeta$,
\begin{enumerate}
\item $L(\xi+\zeta)=\max\cbra L(\xi),L(\zeta)\cket$ and
\item $L(\xi\cdot\zeta)=L(\xi)+L(\zeta)$.
\end{enumerate}
\item The {\em end-logarithm}, {\em last exponent} or {\em least exponent} $\le$ assigns to an ordinal $\xi$ the unique $\delta$ such that $\xi$ can be written in the form $\gamma+\omega^\delta$; we also set $\le 0=0$. Here we see that
\begin{enumerate}
\item $\le(\xi+\zeta)=\le\zeta$ and
\item $\le(\xi\cdot\zeta)=L\xi+\le\zeta$.
\end{enumerate} 
\end{enumerate}

Exponentials and logarithms shall provide us with key examples of functions to hyperate and cohyperate, respectively, and indeed are the authors' motivation for the present work, due to their applicability to provability logic \cite{glpmodels,WellOrders}.

It is convenient to relate the Veblen progression $\langle e_\xi\rangle_{\xi\in{\sf On}}$ with $\langle\varphi_\xi\rangle_{\xi\in{\sf On}}$:

\begin{lemm}\label{exptoveb}
Given ordinals $\alpha,\beta$, $\langle e_\xi\rangle_{\xi\in{\sf On}}={\rm Veb}[e]$ and $\langle \varphi_\xi\rangle_{\xi\in{\sf On}}={\rm Veb}[\varphi]$,
\begin{enumerate}
\item $e_\alpha(0)=0$,
\item $e_0(1+\beta)=\varphi_0(1+\beta)$,
\item $e_{1+\alpha}(1+\beta)=\varphi_{1+\alpha}(\beta)$.
\end{enumerate}
\end{lemm}

\proof\
\paragraph{1.} Since $e_0(0)=0$, $0$ is a fixed point of $e_0$ and, by an easy induction, one can check that it is a fixed point of all $e_\xi$. Thus $e_\xi(0)=0$ for every $\xi\in{\sf On}$.

\paragraph{2.} Obvious from the definition of $e=e_0$.

\paragraph{3.} Let us prove this claim by induction on $\alpha$. Here it suffices to check that, for all $\alpha'<\alpha$, $1+\beta$ is a fixed point of $e_{1+\alpha'}$ if and only if it is a fixed point of $\varphi_{\alpha'}$.

Observe first that all non-zero fixed points of $e_0$ or $\varphi_0$ are infinite, so that $1+\beta=\beta$. Thus
\begin{align*}
e_{\alpha'}(\beta)=\beta&\stackrel{\rm IH}\Leftrightarrow \varphi_{\alpha'}(1+\beta)=\beta\\
&\Leftrightarrow \varphi_{\alpha'}(\beta)=\beta.
\end{align*}
\endproof
\begin{example}\label{newex}
In view of Example \ref{Example:lOmega}, we must have that
\[e_1(2)=\varphi_1(1)=\varepsilon_1.\]
\end{example}

\subsection{Weak normal forms}


Hyperations in general and hyperexponentials in particular can be used to give different sorts of notation system for ordinals.

For example, given an ordinal ${\x}$, we say an expression
\[
{\x}=\sum_{i\leq I}\ex^{\alpha_i}\beta_i+n
\]
is a {\em Weak Hyperexponential Normal Form} if $I,n<\omega$, $\beta_i<\ex^{\alpha_i}\beta_i$ for $i\leq I$ and $\ex^{\alpha_i}\beta_i\geq \ex^{\alpha_{i+1}}\beta_{i+1}$ whenever $i<I$. Note that Weak Hyperexponential Normal Forms are typically not unique. For example, $\omega^\omega = e^21=e^1e^11$. We do, however, have uniqueness if every $\alpha_i$ is of the form $\omega^\delta$.

Say an ordinal ${\x}$ is {\em definable} by a set of ordinals $\Theta$ if ${\x}$ has a normal form $\sum_{i< I}\ex^{\alpha_i}\beta_i+n$ where each $\alpha_i\in\Theta$ and, inductively, $\Theta$ defines each $\beta_i$. Every set of ordinals defines $0$.

\begin{prop}
Every ordinal ${\x}>0$ has a weak hyperexponential normal form and hence is definable by $\Theta$ large enough.

If we further require that every exponent be of the form $\omega^\delta$, then the WHNF obtained is unique.
\end{prop}

\proof
Write ${\x}$ in Veblen Normal Form and, in view of Lemma \ref{exptoveb}, replace $\varphi_\alpha(\beta)$ by $\ex^{\omega^\alpha}(1+\beta)$ for $\alpha >0$, $\varphi_0 (\beta)$ by $e^1(\beta)$ for $\beta>0$. The occurrences of $\varphi_0(0)$ can be captured in the term ${}+n$ in the end of a WHNF.

If all exponents are of the form $\omega^\delta$, we may invert the process to obtain a VNF from a given WHNF; the uniqueness of the latter follows from the uniqueness of the former.
\endproof

This proposition provides us one way to uniquely choose a particular WHNF. In \cite{FernandezJoostenVeblenAndTheWorm2012} it is shown that $\max \{ \alpha \mid \exists \gamma \ e^\alpha \gamma = \omega^\beta \}$ exists for each ordinal $\beta$. This gives rise to another unique WHNF representation where each $\alpha_i$ in ${\x}=\sum_{i< I}\ex^{\alpha_i}\beta_i+n$ is chosen maximal.


\section{Cohyperations}\label{section:Cohyperations}

As mentioned in Section \ref{section:AdditivityVersusCoadditivity}, left-additive iteration does not allow for injectivity and thus is not possible within the class of normal functions. However, left-additive iteration of {\em initial} functions can be defined and indeed is closely tied to hyperation; we will call this form of iteration {\em cohyperation}. Recall that a function is initial if it maps initial segments to initial segments.

Cohyperations are tailored to yield left adjoints to hyperations; thus, as we observed in Lemma \ref{isanadj} we are bound to work with initial functions. Moreover, left-additivity should naturally be associated to cohyperations; indeed, let $\langle g^\xi\rangle_{\xi\in{\sf On}}$ be a family of left inverses to a weak hyperation $\langle f^\xi\rangle_{\xi\in{\sf On}}$ (i.e., $g^\xi f^\xi={\sf id}$ for all $\xi$), and suppose that $\gamma= f^{\alpha+\beta}\delta$. Observe that
\[
g^\beta g^\alpha f^{\alpha + \beta}\delta = g^\beta g^\alpha f^{\alpha} f^{\beta}\delta = g^\beta f^{\beta}\delta = \delta.
\]
But at the same time $\delta=g^{\alpha+\beta}\gamma$, and thus we have that
\[
g^{\alpha+\beta} \gamma = g^\beta g^\alpha \gamma.
\]
Hence in order to produce left adjoints to hyperations we shall first turn our attention to developing a theory of left-additive iterations of initial functions.

\subsection{Weak cohyperations}

Much of the work for cohyperations is analogous to that for hyperations, although there are important differences. As before, we first define {\em weak} cohyperations of a function $f$, and now impose a {\em maximality} condition to pick out the cohyperation of $f$.

\begin{definition}[Cohyperation]
Let $\Lambda$ be either an ordinal or the class of all ordinals and let $f$ be an initial function.

A {\em weak cohyperation} of $f$ is a family of initial functions $\langle g^{\x}\rangle_{\xi<\Lambda}$ such that
\begin{enumerate}
\item $g^0{\x}={\x}$ for all ${\x}$,
\item $g^1=f$,
\item $g^{{\x}+\zeta}=g^\zeta g^{\x}$.
\end{enumerate}

If $\Lambda={\sf On}$ and $g$ is maximal in the sense that $g^{\x}\zeta\geq h^{\x}\zeta$ for every weak cohyperation $h$ of $f$ and all ordinals $\xi,\zeta$, we say $g$ is {\em the cohyperation} of $f$ and write ${\rm coH}[f]=\langle g^\xi\rangle_{\xi\in{\sf On}}$.
\end{definition}

Weak cohyperations also have some nice properties. The proofs of the following two results are very similar to their hyperation analogue and we omit them.

\begin{lemma}[Properties of weak cohyperations]\label{theorem:propertiesCohyperations}
If $\langle g^{\x}\rangle_{\xi<\Lambda}$ is a weak cohyperation of $f$, then
\begin{enumerate}
\item $g^\zeta\alpha\leq g^{\x}\alpha$ whenever ${\x}<\zeta$,
\item $g^{\zeta}\alpha=g^\zeta g^{\x}\alpha$ whenever ${\x}+\zeta=\zeta$.
\end{enumerate}
\end{lemma}

\begin{lemma}\label{cochar}
Let $f$ be an initial function and let $\vec g=\langle g^{\x}\rangle_{\xi<\Lambda}$ be a family of initial functions. Then, the following are equivalent:
\begin{enumerate}
\item $\vec g$ is a weak cohyperation of $f$;
\item $\vec g$ satisfies the following:
\begin{enumerate}
\item $g^0\alpha=\alpha$,
\item $g^1=f$,
\item $g^{\omega^\rho+{\x}}=g^{{\x}}g^{\omega^\rho}$ provided ${\x}<\omega^\rho+{\x}$,
\item $g^{\omega^\rho}{\x}=g^{\omega^\rho}g^{\omega^\delta}{\x}$ whenever $\delta<\rho$.
\end{enumerate}
\end{enumerate}
\end{lemma}

The clause $g^{\omega^\rho}{\x}=g^{\omega^\rho}g^{\omega^\delta}{\x}$ will imply that computing $g^{\omega^\rho}{\x}$ can often be reduced to computing $g^{\omega^\rho}g^\eta{\x}$, for appropriately chosen $\eta$. The next lemma will give us an especially convenient candidate.

Given ordinals $\alpha,\x$ and a family of functions $\vec f=\langle f^\eta\rangle_{\eta<\x}$, we define $\eta^\ast(\vec f,\xi,\alpha)$ as the least ordinal $\eta^\ast<\xi$ such that
\[f^{\eta^\ast}\alpha=\min_{\eta\in [0,\xi)}f^\eta\alpha.\]

Thus, $\eta^\ast$ is the least value which minimizes $f^\eta\alpha$. In particular, if $f^0$ is the identity, then $\eta^\ast>0$ if and only if $f^{\eta^\ast}\alpha<\alpha$.

\begin{lemm}\label{stable}
Given a weak cohyperation $\vec g=\langle g^{\x}\rangle_{\xi<\Lambda}$, ordinals $\xi,\alpha$ and $\eta^\ast=\eta^\ast(\vec g,\xi,\alpha)$,
\begin{enumerate}
\item $g^\eta\alpha>g^{\eta^\ast}\alpha$ whenever $\eta<\eta^\ast$ and
\item $g^\eta\alpha=g^{\eta^\ast}\alpha$ whenever $\eta^\ast<\eta<\xi$.
\end{enumerate}
\end{lemm}

\proof
The first claim follows from the minimality of $\eta^\ast$ and indeed does not depend on left additivity.

Meanwhile, if $\eta\in[\eta^\ast,\xi)$, we use Lemma \ref{theorem:propertiesCohyperations}.1 to see that $g^\eta\alpha\leq g^{\eta^\ast}\alpha$, but by the minimality of $g^{\eta^\ast}\alpha$, equality must hold.
\endproof

\subsection{Cohyperations}
With the definition of $\eta^\ast$ we can give a recursive definition for a weak cohyperation for an initial function $f$ which, we shall see, gives us {\em the} cohyperation of $f$. 
 However, until we prove this fact we shall call it the {\em recursive cohyperation} of $f$:

\begin{definition}[Recursive cohyperation]
Given an initial function $f$, we define a sequence of functions ${\rm coH}'[f]=\langle f^{\xi}\rangle_{\xi\in\mathsf{On}}$ by the following recursion:
\begin{enumerate}
\item $f^{0}\alpha=\alpha$,
\item $f^{\xi} 0=0$,
\item $f^1=f$,
\item $f^{\omega^\rho+\x}=f^{\x}f^{\omega^\rho}$ provided $\omega^\rho+\x<\x$,
\item if $\rho>0$ and $\eta^\ast=\eta^\ast({\rm coH}'[f],\omega^\rho,\alpha)>0$, then 
\[f^{\omega^\rho}{\alpha}=f^{\omega^\rho}f^{\eta^\ast} {\alpha},\]
\item if $\rho>0$ and $\eta^\ast({\rm coH}'[f],\omega^\rho,\alpha)=0$, then
\[f^{\omega^\rho}{\alpha}=\sup_{\beta\in[0,\alpha)}(f^{\omega^\rho}(\beta)+1).\]
\end{enumerate}

\end{definition}

Later we will show that ${\rm coH}'[f]={\rm coH}[f]$, but this will require some work.

First, let us include a simple example to see the recursive definition at work. Recall from Lemma \ref{theorem:BasicPropertiesOrdinalArithmetic} that $\le(\alpha)$ denotes the last exponent of $\alpha$ when written in CNF.

Below, write ${\rm coH}'[\ell]=\langle\ell^\xi\rangle_{\xi\in{\sf On}}$.

\begin{example}\label{Example:lOmega}
$\le^{\omega}\varepsilon_1 = 2$.
\end{example}

\begin{proof}
It is not hard to see that $\le$ is indeed an initial function.  We will first compute $\le^{\omega}\varepsilon_0$. Clearly, for each $n<\omega$ we have that $\le^{n}\varepsilon_0=\varepsilon_0$ whence $\eta^\ast({\rm coH}'[\ell],\omega, \varepsilon_0) = 0$ and
\[
\le^{\omega}\varepsilon_0 = \sup_{\beta\in [0,\varepsilon_0)} (\le^{\omega}(\beta) +1).
\]
Now, for any $\beta < \varepsilon_0$, there is some $n<\omega$ with $\le^{n}\beta = 0$, whence 
\[
\le^{\omega}\beta = \le^{n+ \omega}\beta = \le^{\omega}\le^{n}\beta = \le^{\omega}0 =0.
\]
Consequently, $\le^{\omega}\varepsilon_0 = \sup_{\beta\in [0,\varepsilon_0)} (\le^{\omega}(\beta) +1) = 1$.

Now we can compute $\le^{\omega}\varepsilon_1$. Again, for all $n<\omega$, $\le^{n}\varepsilon_1=\varepsilon_1$ so that $\eta^\ast({\rm coH}'[\ell],\omega, \varepsilon_1) = 0$ and 
\[
\le^{\omega}\varepsilon_1 = \sup_{\beta\in [0,\varepsilon_1)} (\le^{\omega}(\beta) +1).
\]
So, we first need to compute $\le^{\omega}\beta$ for any $\beta \in (\varepsilon_0,\varepsilon_1)$. A straightforward induction shows that for each such $\beta$ we can find a large enough $n<\omega$ so that either $\le^{n}\beta = 0$ or $\le^{n}\beta = \varepsilon_0$, given that $\ell$ has no fixpoints in this interval. Thus, $\le^{\omega}\beta = \le^{\omega}\le^{n}\beta$ which is either $0$ or $\le^{\omega}\varepsilon_0 = 1$. Consequently, $\le^{\omega}\varepsilon_1 = \sup_{\beta\in [0,\varepsilon_1)} (\le^{\omega}(\beta) +1) = 2$.
\end{proof}

The next few lemmas will be used to prove that ${\rm coH}'[f]={\rm coH}[f]$.

\begin{lemma}\label{isinitial}
If $f$ is an initial function and $\langle f^\xi\rangle_{\xi\in{\sf On}}={\rm coH}'[f]$, then $f^\xi$ is an initial function for all $\x$.
\end{lemma}

\proof
The identity, i.e. $f^{0}$, clearly is an initial function. We proceed by induction on $\x$, assuming that $f^{\x'}$ is an initial function for all $\x'<\x$.

For the case $\x=\gamma+\delta$ with $\gamma,\delta<\xi$, a very easy argument shows that the composition of initial functions is initial, so $f^{\x}$ is initial. 

Now, assume $\x=\omega^\rho$. Let us check that $f^{\omega^\rho}$ is initial. Assume $f^{\eta}$ is initial for all $\eta<\omega^\rho$. We will prove, by a subsidiary induction on $\zeta$, that $f^{\omega^\rho}[0,\zeta]$ is an initial segment, assuming that, given $\zeta'<\zeta$, $f^{\omega^\rho}[0,\zeta']$ is an initial segment.

First we note that,
\[
f^{\omega^\rho}[0,\zeta]=\cbra  f^{\omega^\rho}\zeta\cket\cup\bigcup_{\zeta'<\zeta}f^{\eta}[0,\zeta'],
\]
and by induction on $\zeta$, $\bigcup_{\zeta'<\zeta}f^{\eta}[0,\zeta']$
is an initial segment, given that it is a union of initial segments; call it $[0,\beta)$.

Now, if $\eta^\ast=\eta^\ast({\rm coH}'[f],\omega^\rho,\zeta)>0$, we see that $f^{\eta^\ast}\zeta<\zeta$ and thus
\[f^{\omega^\rho}\zeta=f^{\omega^\rho}f^{\eta^\ast}\zeta\in[0,\beta),\]
so that $f^{\omega^\rho}[0,\zeta]=[0,\beta)$, an initial segment.

Otherwise,
\[
f^{\omega^\rho}\zeta=\sup_{\vartheta\in[0,\zeta)} (f^{\omega^\rho}(\vartheta)+1)=\beta.
\]

Thus
\[f^{\omega^\rho}[0,\zeta]=[0,\beta],\]
which is also an initial segment, as claimed.
\endproof

\begin{lemma}\label{isweakco}
If $f$ is an initial function, then $\langle f^\xi\rangle_{\xi\in{\sf On}}={\rm coH}'[f]$ is a weak cohyperation of $f$.
\end{lemma}

\proof
Let us show by induction on $\Lambda$ that $\langle f^{\xi}\rangle_{\xi<\Lambda}$ is a weak cohyperation.

We have seen in Lemma \ref{isinitial} that $f^{\xi}$ is initial for all $\xi$. Thus, in view of Lemma \ref{cochar}, we must check that 
\begin{equation}\label{firstco}
f^{\omega^\rho+{\x}}=f^{{\x}}f^{\omega^\rho}
\end{equation}
provided ${\x}<\omega^\rho+{\x}<\Lambda$ and
\begin{equation}\label{secondco}
f^{\omega^\rho}{\x}=f^{\omega^\rho}f^{\omega^\delta}{\x}
\end{equation}
whenever $\omega^\delta<\omega^\rho<\Lambda$. We may assume, inductively, that $\langle f^{\xi}\rangle_{\xi<\omega^\rho}$ is a weak cohyperation.

Note, however, that (\ref{firstco}) is satisfied by the definition of $f^{\omega^\rho+\xi}$, so it suffices to show (\ref{secondco}). More generally we shall show that, if $\eta<\omega^\rho$ and $\zeta$ is any ordinal, $f^{\omega^\rho}\zeta=f^{\omega^\rho}f^{\eta}\zeta$.

If $f^{\eta}\zeta=\zeta$, there is nothing to prove. Hence we are left with the case $f^{\eta}\zeta<\zeta$, which in particular implies that $\eta^\ast=\eta^\ast({\rm coH}'[f],\omega^\rho,\zeta)$ is non-zero.

Here we must consider two subcases. The case $\eta>\eta^\ast$ is easy, since $f^{\eta^\ast}\zeta=f^{\eta}\zeta$ and by definition $f^{\omega^\rho}\zeta=f^{\omega^\rho}f^{\eta^\ast}\zeta$; thus we focus on the case $\eta<\eta^\ast$.

By induction on $f^\eta\zeta<\zeta$, we may assume that, given $\vartheta<\omega^\rho$,
\[
f^{\omega^\rho}f^{\eta}\zeta= f^{\omega^\rho}f^{\vartheta}f^{\eta}\zeta=f^{\omega^\rho}f^{\eta+\vartheta}\zeta,
\]
where the second equality uses the hypothesis that $\langle f^{\xi}\rangle_{\xi<\omega^\rho}$ is a weak cohyperation.

In particular, for $\vartheta=-\eta+\eta^\ast$ we see that
\[f^{\omega^\rho}f^{\eta}\zeta= f^{\omega^\rho}f^{\eta^\ast}\zeta=f^{\omega^\rho}\zeta.\]
Setting $\eta=\omega^\delta$ we obtain (\ref{secondco}).
\endproof

Now that we see that the recursive cohyperation actually defines a cohyperation we are allowed to replace the clauses 5 and 6 respectively by\\\\
\begin{tabular}{ll}
5'. & {If $\rho>0$ and $f^{\eta} {\alpha} < \alpha$ for some $\eta<\omega^\rho$, then} $f^{\omega^\rho}{\alpha}=f^{\omega^\rho}f^{\eta} {\alpha}$,\\\\
6'. &
If $\rho>0$ and $f^{\eta} {\alpha} = \alpha$ for all $\eta< \omega^\rho$, then
\end{tabular}
\[f^{\omega^\rho}{\alpha}=\sup_{\beta\in[0,\alpha)}(f^{\omega^\rho}(\beta)+1).\]

We are almost ready to prove that that the recursive cohyperation actually defines the unique cohyperation, but we first need a technical lemma.

\begin{lemm}\label{lessthan}
Let $f$ be an initial function and $\langle f^\xi\rangle_{\xi\in{\sf On}}={\rm coH}'[f]$.

If $\eta^\ast=\eta^\ast({\rm coH}'[f],\omega^\rho,\alpha)$ and $\beta< f^{\eta^\ast}\alpha$ then $f^{\omega^\rho}\beta< f^{\omega^\rho}\alpha.$
\end{lemm}

\proof
First assume that $\eta^\ast=0$, that is, $f^{\eta}\alpha=\alpha$ for all $\eta<\omega^\rho$.

Then, by Clause 6,
\[f^{\omega^\rho}{\alpha}=\sup_{\gamma<\alpha}(f^{\omega^\rho}(\gamma)+1),\]
so in particular $f^{\omega^\rho}\beta< f^{\omega^\rho}\alpha$.

We now consider the case that $\eta^\ast >0$ and note that
\[
f^{\eta}f^{\eta^\ast}\alpha=f^{\eta^\ast+\eta}\alpha=f^{\eta^\ast}\alpha
\]
for all $\eta<\omega^\rho$. But this shows that $\eta^\ast({\rm coH}'[f],\omega^\rho,f^{\eta^\ast}\alpha)=0$.

It follows from the first case that, if $\beta<f^{\eta^\ast}\alpha$,
\[
f^{\omega^\rho}\beta< f^{\omega^\rho}f^{\eta^\ast}\alpha=f^{\omega^\rho}\alpha.
\]
\endproof
We are now ready to prove our main theorem.

\begin{theorem}
If $f$ is an initial function then ${\rm coH}'[f]$ is its cohyperation, that is, \[{\rm coH}'[f]={\rm coH}[f].\]
\end{theorem}

\proof
Let $\langle f^\xi\rangle_{\xi\in{\sf On}}={\rm coH}'[f]$.

In view of Lemma \ref{isweakco}, all that remains for us to prove is that
 $f^{\xi}$ bounds every weak cohyperation $g^\xi$, i.e., that $f^{\xi}\zeta\geq g^\xi\zeta$ for all $\xi,\zeta$.

Now we work by induction on $\xi$ and $\zeta$, considering several cases depending of which clause defines $f^{\xi}\alpha$.

\paragraph{Clauses 1-3.} These clauses establish the desired inequality for $\xi\leq 2$.

\paragraph{Clause 4.} Write $\x=\x'+\omega^\rho$ and let $\eta^\ast=\eta^\ast({\rm coH}'[f],\omega^\rho,f^{\xi'}\alpha)$. We note that
\[
f^{\omega^\rho}f^{\x'}\alpha =f^{\omega^\rho}f^{\eta^\ast}f^{\xi'}\alpha=f^{\omega^\rho} f^{\xi'+\eta^\ast}\alpha
\]
and further
\[
\eta^\ast({\rm coH}'[f],\omega^\rho,f^{\xi'+\eta^\ast}\alpha)=0.
\]
Now, we have by induction on $\xi'+\eta^\ast<\omega^\rho$ that $f^{\xi'+\eta^\ast} \alpha\geq g^{\xi'+\eta^\ast}\alpha$, and thus we can use Lemma \ref{lessthan} to see that
\[
f^{\omega^\rho}f^{\eta^\ast}f^{\xi'}\alpha=f^{\omega^\rho}f^{\xi'+\eta^\ast} \alpha\geq f^{\omega^\rho}g^{\xi'+\eta^\ast}\alpha\stackrel{\mathrm{IH}}\geq g^{\omega^\rho}g^{\xi'+\eta^\ast}\alpha=g^\x\alpha.
\]

\paragraph{Clause 5.} In this case, $\xi=\omega^\rho$ and we have that
\[
f^{\omega^\rho}\zeta=f^{\omega^\rho}f^{\eta^\ast}\zeta,
\]
where $\eta^\ast=\eta^\ast({\rm coH}'[f],\omega^\rho,\alpha)>0$.

By induction we have $g^{\eta^\ast}\alpha\leq f^{\eta^\ast}\alpha$ and we use Lemma \ref{lessthan} to see that
\[f^{\omega^\rho}\alpha\geq f^{\omega^\rho}g^{\eta^\ast}\alpha\stackrel{\mathrm{IH}}\geq g^{\omega^\rho}g^{\eta^\ast}\alpha=g^{\omega^\rho}\alpha.\]

\paragraph{Clause 6.} We see that
\[f^\xi\alpha=\sup_{\beta<\alpha}(f^\xi(\beta)+1)\stackrel{\rm IH}\geq \sup_{\beta<\alpha}(g^\xi(\beta)+1)\geq g^\xi\alpha,\]
where the last inequality uses the assumption that $g^\xi$ is initial.
\endproof

Like hyperations, cohyperations of cohyperations behave as one would expect. We present the following without proof, which is very similar to its hyperation analogue.

\begin{lemm}
If $f$ is an initial function and $\xi,\zeta$ are ordinals then $(f^{\xi})^\zeta=f^{\xi\cdot\zeta}$.
\end{lemm}


\section{Exact sequences}\label{exseq}

One nice thing about cohyperations is that, in a sense, they need only be defined locally. To make this precise, we need the notion of an {\em exact sequence}.

\begin{definition}
Let $g$ be an initial function, $\langle g^{\x}\rangle_{\xi\in{\sf On}}={\rm coH}[g]$ and $f:\Lambda\to\Theta$ be an ordinal function.

Then, we say $f$ is {\em $g$-exact} if, given ordinals ${\x},\zeta$, $f({\x}+\zeta)=g^\zeta f({\x})$.
\end{definition}

A $g$-exact function $f$ describes the values of $g^{\x} f(0)$. However, for $f$ to be $g$-exact, we need only check a fairly weak condition:

\begin{lemma}\label{tfae}
Let $g$ be an initial function and $\langle g^{\x}\rangle_{\xi\in{\sf On}}={\rm coH}[g]$.

Then, the following are equivalent:
\begin{enumerate}
\item $f$ is $g$-exact
\item for every ordinal ${\x},$ $f({\x})=g^{\x} f(0)$
\item for every ordinal $\zeta>0$ there is ${\x}<\zeta$ such that $f(\zeta)=g^{-{\x}+\zeta}f({\x})$.
\end{enumerate}
\end{lemma}
\proof
It should be clear that condition 1 implies the other two, so we shall show that 3 implies 2 and 2 implies 1.

\paragraph{$\bf 3\Rightarrow 2$.} Let $f$ satisfy condition 3, and suppose inductively that, for all $\xi<\zeta$, $f(\xi)=g^\xi f(0)$.

Pick $\xi<\zeta$ such that $f(\zeta)=g^{-\xi+\zeta}f(\xi)$. Then,
\[f(\zeta)=g^{-\xi+\zeta}f(\xi)=g^{-\xi+\zeta}g^\xi f(0)=g^\zeta f(0),\]
as claimed.

\paragraph{$\bf 2\Rightarrow 1$.} Assume that $f$ satisfies condition 2. We then see that, given $\xi<\zeta$,
\[
f(\zeta)=g^\zeta  f(0)=g^{-\xi+\zeta}g^\xi f(0)= g^{-\xi+\zeta}f(\xi).
\]
\endproof

\begin{example}
The sequence
\[
(\varepsilon_1, \varepsilon_1, \varepsilon_1, \ldots, 2, 0, \ldots)
\]
that is constant $\varepsilon_1$ for the first $ \omega$ coordinates, $2$ on the $\omega^{\rm th}$ coordinate and $0$ after that is $\le$-exact.
\end{example}

\begin{proof}
By Example \ref{Example:lOmega} and Lemma \ref{tfae}.
\end{proof}

\section{Inverting hyperations}\label{ih}

As promised, cohyperation provides left-inverses for hyperations. In this section we will show how this is obtained.

\begin{theorem}\label{leftinv}
Suppose that $f$ is a normal function and $\langle f^\xi\rangle_{\xi\in{\sf On}}$ its hyperation. Let $g$ be a left adjoint for $f$ with cohyperation $\langle g^\xi\rangle_{\xi\in{\sf On}}$.

Then, $g^\xi$ is a left adjoint of $f^\xi$ for all $\xi$.
\end{theorem}

\proof
We prove this by induction on $\x$.

Suppose that $\alpha\leq f^\xi\beta$. If $\x=\gamma+\delta$ with $\gamma,\delta<\xi$, we see that $\alpha\leq f^\gamma f^\delta\beta$, so that $g^\gamma \alpha\leq f^\delta\beta $ and thus $g^\delta g^\gamma\alpha\leq \beta$. Strict equality and strict inequality are preserved by this argument.

Thus we may assume $\xi=\omega^\rho$ for some $\rho>0$.

The case $\beta=0$ is also easy; recall that by Lemma \ref{iszero}, $f(0)=0$, and thus $f^\xi 0=0$ for all $\xi$. But then, $g^\xi 0$ is also $0$ for all $\xi$ because $g^\xi$ is initial.

So we assume that $\beta\not=0$. Let us also suppose, inductively, that (\ref{laone}) and (\ref{latwo}) hold for $\alpha',\beta'$ whenever ({\em i}) $\beta'<\beta$ or ({\em ii}) $\beta'=\beta$ and $\alpha'<\alpha$.

Consider first the case that equality holds; $\alpha=f^{\omega^\rho}\beta$. Note that, if $\eta<\omega^\rho$, by induction on $\xi=\omega^\rho$,
\[g^\eta f^{\omega^\rho}=g^\eta f^\eta f^{-\eta+\omega^\rho}=f^{-\eta+\omega^\rho}=f^{\omega^\rho};\]
thus given $\beta$, there is no $\eta<\omega^\rho$ with $g^\eta f^{\omega^\rho}\beta<f^{\omega^\rho}\beta.$

It follows that $g^{\omega^\rho}\alpha$ is defined by Clause 6; that is,
\[g^{\omega^\rho}\alpha =\sup_{\alpha'\in [0,\alpha)} (g^{\omega^\rho}(\alpha')+1)\leq \beta,\]
where the inequality follows by induction on $\alpha'<\alpha$, since if $\alpha'<\alpha= f^{\omega^\rho}\beta$, $g^{\omega^\rho}\alpha'<\beta$.

Next we check that $g^{\omega^\rho}\alpha\geq \beta$.

For all $\beta'<\beta$, we can assume by induction that $g^{\omega^\rho}f^{\omega^\rho}\beta'=\beta'$. But $f^{\omega^\rho}\beta'<f^{\omega^\rho}\beta=\alpha$, and thus
\[g^{\omega^\rho}\alpha\geq \sup_{\beta'<\beta}(g^{\omega^\rho} f^{\omega^\rho}(\beta')+1)\stackrel{\mathrm{IH}}=\sup_{\beta'<\beta}(\beta'+1) =\beta.\]

Finally, let us see that if $\alpha<f^{\omega^\rho},$ then indeed $g^{\omega^\rho}\alpha<\beta$. Here we proceed by induction on $\beta$.

If $\beta$ is a limit ordinal, then $\alpha< f^{\omega^\rho}\beta'$ for some $\beta'<\beta$, in which case by induction hypothesis $g^{\omega^\rho}\alpha< \beta'<\beta$.

If $\beta=\beta'+1$, $\alpha < f^\eta(f^{\omega^\rho}(\beta')+1)$ for some $\eta<\omega^\rho$. Thus, by induction on $\eta<\xi=\omega^\rho$, $g^{\eta}\alpha<f^{\omega^\rho}(\beta')+1$, so that $g^\eta\alpha\leq f^{\omega^\rho}\beta'$ and by induction on $\beta'<\beta$,
\[g^{\omega^\rho}\alpha=g^{\omega^\rho}g^{\eta}\alpha\leq \beta'<\beta.\]

\endproof

\begin{cor}
If $f$ is a normal function with hyperation $\langle f^\xi\rangle_{\xi\in{\sf On}}$ and $g$ is a left adjoint to $f$ with cohyperation $\langle g^\xi\rangle_{\xi\in{\sf On}}$, then, given ordinals $\xi<\zeta$ and $\alpha$, $g^\xi f^\zeta=f^{-\xi+\zeta}$ and $g^\zeta f^\x=g^{-\xi+\zeta}$.
\end{cor}

\proof
Immediate from Theorem \ref{leftinv}, writing $f^\zeta=f^\xi f^{-\xi+\zeta}$ and $g^\zeta=g^{-\xi+\zeta}g^\x$.
\endproof

As an important example we mention {\em hyperlogarithms}, i.e., the cohyperations $\langle L^\xi\rangle_{\xi\in{\sf On}}$ and $\langle \ell^\xi\rangle_{\xi\in{\sf On}}^\xi$, where, recall, $L(\omega^\alpha+\beta)=\alpha$ (provided $\beta<\omega^\alpha+\beta$) and $\le(\alpha+\omega^\beta)=\beta$.

\begin{corollary}\label{lastcor}
The cohyperations $L^\xi,\le^\xi$ are left adjoints to the hyperations $\ex^\x$ for all $\x$.
\end{corollary}

\proof
Immediate from Theorem \ref{leftinv} and the fact that $L,\ell$ are both initial left adjoints to $e$.\endproof

\begin{example}
As seen in Example \ref{newex}, $e^\omega 2=\varepsilon_1$; meanwhile, according to Example \ref{Example:lOmega}, $\ell^\omega\varepsilon_1=2$.

Then, $\ell^\omega e^\omega 2=2$, as predicted by Corollary \ref{lastcor}.
\end{example}


\section*{Concluding remarks}

The work we have presented gives a general theory of two forms of iteration of ordinal functions which give natural alternatives to existing notions. The ordinals one may generate through hyperations are no larger than those given by Veblen progressions. However, the deep algebraic structure of hyperations and cohyperations has great advantages in applications, as the authors have found in their work in provability logics, and we trust that many mathematicians working routinely with ordinals will find our framework appealing.

As with Veblen progressions, one may work with higher-order versions of hyperations, where iteration is applied to exponents of functions, thus facilitating operations well beyond $\Gamma_0$. This is a direction well worth looking into, and possibly a necessary one if provability logics are to work their way into the analysis stronger and stronger theories.

Moreover, one may also think of hyperations and cohyperations as special cases of iteration through additive sequences of functions. To be precise, one may think of hyperations (and similarly cohyperations) as operators of the form ${\rm Hyp}^F_{\mathcal C}$, where $\mathcal C$ is a class of functions and $F$ a criterion for selecting a preferred candidate from all right-additive function families in $\mathcal C$. From this perspective, ${\rm Hyp}$ would become ${\rm Hyp}^{\sf min}_{\sf Nrm}$, where $\sf Nrm$ denotes the class of normal functions and $\sf min$ denotes pointwise minimization. We have studied the two versions which were most useful for a specific application; however, there are many more possibilities left to be explored.

\section*{Acknowledgments}
This will be filled out later.

\bibliographystyle{plain}
\bibliography{biblio}







\end{document}